\documentclass[reqno]{amsart}

\usepackage{tabu}
\usepackage{amssymb}
\usepackage{mathtools}
\usepackage{a4wide,amsmath}
\usepackage{mathrsfs}
\usepackage{amsthm}
\numberwithin{equation}{section}
\numberwithin{figure}{section}
\numberwithin{table}{section}
\usepackage{bbm}
\usepackage{subfig}
\usepackage{enumerate}
\usepackage{needspace}
\usepackage[section]{placeins}
\usepackage{graphicx}		  
\usepackage{ifpdf}
\ifpdf
\DeclareGraphicsExtensions{.pdf,.eps,.jpg,.png}	
\usepackage[suffix=]{epstopdf}
\fi
\usepackage{xcolor}
\usepackage[utf8]{inputenc}
\usepackage{hyperref}
\hypersetup{hidelinks}

\long\def\MSC#1\EndMSC{\def\arg{#1}\ifx\arg\empty\relax\else
	{\narrower\noindent%
		{2020 Mathematics Subject Classification}: #1\\} \fi}
\long\def\PACS#1\EndPACS{\def\arg{#1}\ifx\arg\empty\relax\else
	{\narrower\noindent%
		{PACS numbers}: #1}\fi}
\long\def\KEY#1\EndKEY{\def\arg{#1}\ifx\arg\empty\relax\else
	{\narrower\noindent%
		Keywords: #1\\}\fi}

\theoremstyle{plain}
\newtheorem{theorem}{Theorem}[section]

\newtheorem{proposition}[theorem]{Proposition}
\newtheorem{corollary}[theorem]{Corollary}
\theoremstyle{definition}
\newtheorem{definition}[theorem]{Definition}

\newtheorem{assumption}[theorem]{Assumption}
\theoremstyle{remark}
\newtheorem{remark}[theorem]{Remark}

\newcommand{\norm}[1]{\lVert#1\rVert}
\newcommand{\abs}[1]{\lvert#1\rvert} 
\newcommand{\inner}[1]{\langle#1\rangle} 
\newcommand{\redel}{\mathop{\textup{Re}}}
\newcommand{\imdel}{\mathop{\textup{Im}}}

\newcommand{\suppm}{\mathop{\textup{supp}}}
\newcommand{\R}{\mathbb{R}}
\newcommand{\N}{\mathbb{N}}
\newcommand{\C}{\mathbb{C}}

\newcommand{\I}{\mathrm{i}}    
\newcommand{\di}{\mathrm{d}}   

\newcommand{\upR}{^{\textup{R}}}
\newcommand{\upI}{^{\textup{I}}}
\newcommand{\uptR}{^{\!\textup{R}}}
\newcommand{\uptI}{^{\!\textup{I}}}
\newcommand{\LambdaR}{\Lambda\uptR}
\newcommand{\LambdaI}{\Lambda\uptI}
\newcommand{\AR}{A\upR}
\newcommand{\AI}{A\upI}
\newcommand{\Hsa}{\mathcal{H}_{\textup{SA}}}
\newcommand{\DLambda}{\textup{D}\mkern-1.5mu \Lambda}

\begin{document}
	\title[Reconstruction of non-self-adjoint inclusions]{Reconstruction of non-self-adjoint anisotropic and complex inclusions in the Calder\'on problem}
	
	\author[H.~Garde]{Henrik~Garde}
	\address[H.~Garde]{Department of Mathematics, Aarhus University, Aarhus, Denmark.}
	\email{garde@math.au.dk}
	
	\author[D.~Johansson]{David~Johansson}
	\address[D.~Johansson]{Department of Mathematics, Aarhus University, Aarhus, Denmark.}
	\email{johansson@math.au.dk}
	
	\author[T.~Zacharopoulos]{Thanasis~Zacharopoulos}
	\address[T.~Zacharopoulos]{Department of Mathematics, Aarhus University, Aarhus, Denmark.}
	\email{thanzacharop@math.au.dk}
	
	\begin{abstract}
		We generalize recent results on the monotonicity method, for inclusion detection in the partial data anisotropic Calder\'on problem, to very general non-self-adjoint perturbations. This involves a forward model that accounts for both the anisotropic real conductivity and the anisotropic permittivity, and the results hold in any spatial dimension $d\geq 2$. We assume that the inclusion boundaries can be reached from the domain boundary via a set on which the background conductivity is self-adjoint, and that a definiteness condition holds near the inclusion boundaries. Away from the inclusion boundaries we allow general $L^\infty$ non-self-adjoint perturbations. We only require unique continuation based on the self-adjoint part of the background conductivity, thus making the methods compatible with generic unique continuation results.   
	\end{abstract}	
	\maketitle
	
	\KEY
	anisotropic Calder\'on problem, non-self-adjoint, inclusion detection, monotonicity method.
	\EndKEY
	
	\MSC
	35R30, 35R05, 47H05.
	\EndMSC
	
	\section{Introduction} \label{sec:intro}
	
	We consider the partial data anisotropic Calder\'on problem of reconstructing information about a matrix-valued conductivity coefficient $A$ in the equation
	\begin{equation*}
		-\nabla \cdot(A\nabla u) = 0 \text{ in } \Omega,
	\end{equation*}
	based on the Neumann-to-Dirichlet (ND) map $\Lambda(A)$ on an open subset $\Gamma\subseteq \partial\Omega$. Specific to the anisotropic problem, there are significant obstructions to uniqueness in determining a coefficient $A$ from $\Lambda(A)$, see e.g.~\cite{Astala2005,Sylvester1990,Lee1989}. To our knowledge, there are no general characterizations of the non-uniqueness in the non-self-adjoint setting of the anisotropic Calder\'on problem. Instead we determine the outer shape $D^\bullet$ (Definition~\ref{def:outershape}) of inclusions $D = \suppm(A_D-A_0)$, from knowledge of a ``background conductivity'' coefficient $A_0$ and the ND map $\Lambda(A_D)$ for an unknown coefficient $A_D$. A reconstruction method was recently proved in \cite{GJZ2025} for very general $L^\infty$ \emph{self-adjoint} coefficients based on the monotonicity method. The main assumptions are a unique continuation principle (UCP) satisfied by $A_0$, and that $A_D-A_0$ satisfies a definiteness condition near $\partial D^\bullet$.
	
	We now consider the more complicated case of non-self-adjoint inclusions. Namely we write a general conductivity coefficient as
	\begin{equation*}
		A = \AR + \I \AI,
	\end{equation*}   
	where $\AR$ is the \emph{self-adjoint} part of $A$ while $\I\AI$ is the \emph{skew-adjoint} part of $A$. In particular, both $\AR$ and $\AI$ are self-adjoint as matrices. Suppose that both $\AR$ and $\AI$ have real-valued matrix entries. This corresponds to a mathematical model with
	\begin{equation*}
		\AR = \sigma_\omega \quad \text{and} \quad \AI = \omega\varepsilon_\omega,
	\end{equation*}
	where $\sigma_\omega$ is the anisotropic \emph{real} conductivity and $\varepsilon_\omega$ is the anisotropic permittivity, both in general depending on the frequency $\omega$. In this setting, $A$ would be the anisotropic \emph{complex} conductivity (also called an anisotropic admittivity).
	
	We give a number of generalizations of results from \cite{GJZ2025} to the non-self-adjoint setting, using monotonicity inequalities for the self-adjoint parts of the ND maps (Theorem~\ref{thm:generalmono}), but with some limitations on the skew-adjoint part of the background conductivity $A_0$. Theorems~\ref{thm:mainnonlinear} and~\ref{thm:mainlinear} give exact reconstruction methods of $D^\bullet$ provided that: 
	
	\begin{itemize}
		\item a definiteness condition is satisfied near $\partial D^\bullet$ for the self-adjoint parts $\AR_D-\AR_0$, with a sufficiently small skew-adjoint part near $\partial D^\bullet$,
		\item $\AR_0$ satisfies the UCP (see e.g.~\cite[definition 3.2]{GJZ2025}),
		\item all of $\partial D^\bullet$ can be reached from the measurement boundary $\Gamma$ via an open connected set on which $A_0$ is self-adjoint.
	\end{itemize}   
	
	In particular, inside $D$ (away from $\partial D^\bullet$) we allow $A_D - A_0$ to be a general $L^\infty$ non-self-adjoint perturbation. Moreover, no explicit regularity is required from $\partial D^\bullet$. Theorem~\ref{thm:mainlinear} is particularly useful for numerical implementation, since the test-operators are linearized and therefore can be precomputed. This leads to a real-time reconstruction method (see e.g.~\cite{Garde_2017a,GardeStaboulis2016,GardeStaboulis2019}).
	
	The requirement that $A_0$ is self-adjoint on a set that connects all of $\partial D^\bullet$ to $\Gamma$ is related to the non-optimality of the monotonicity inequalities in Theorem~\ref{thm:generalmono}. There are terms such as $\AI_j(\AR_1)^{-1}\AI_j$ for $j\in\{1,2\}$ that do not vanish even if $A_1=A_2$, meaning that the lower and upper bounds in those inequalities are non-zero in the non-self-adjoint case even when the coefficients agree. This leads to the need of some trickery, including more complicated test-operators, in order to make the \emph{outer approach} of the monotonicity method work, and it seems to prevent very general use of the \emph{inner approach} of the monotonicity method in this setting (see \cite{GJZ2025} for the differences between the inner and outer approach for the self-adjoint setting). 
	
	While the monotonicity inequalities and localized-potential results recalled in Section~\ref{sec:tools} were established in~\cite{GJZ2025}, the reconstruction criteria obtained here require new test-operators and contrast conditions to handle the non-self-adjoint contributions. In particular, there is now an additional term in one of the test-operators with a perturbation \emph{outside} the test-inclusion, which is not an obvious choice in terms of the monotonicity method, as the test-operators have previously only been related to perturbations on the test-inclusion itself.
	
	It is worth noting that we only require a UCP result for $\AR_0$ and not $A_0$ itself, which is very useful since most UCP results are proven for self-adjoint coefficients. In particular, for spatial dimension $d=2$, for any background conductivity $A_0$ for which $\AR_0$ is real-valued, then $\AR_0$ satisfies the UCP~\cite{Alessandrini2012}, while for spatial dimension $d\geq 3$, a sufficient condition is to also require that $\AR_0$ is Lipschitz regular \cite{Garofalo1986,Mandache1998,Miller74}. 
	
	In the case when $A_0$ is self-adjoint and it is only the perturbation $A_D-A_0$ that is non-self-adjoint, the assumptions and test-operators simplify, and the corresponding results are given in Corollary~\ref{coro}. In \cite[theorem 12.2]{GJZ2025} a general result was given using \emph{extreme} test-operators, based on perfectly conducting and perfectly insulating parts of a coefficient. This method does not require bounds on the magnitude of the unknown perturbation. In Theorem~\ref{thm:mainextreme} we show that the exact same extreme test-operators that work for very general self-adjoint perturbations also work for very general non-self-adjoint perturbations, where we again assume that $A_0$ is self-adjoint. The proof becomes very short by rewriting the problem in a way that allows utilizing \cite[theorem~12.2]{GJZ2025} based on self-adjoint coefficients.
	
	We emphasize that applications include self-adjoint materials (vanishing permittivity) in the background, which is the more traditional setting for electrical impedance tomography, but where the unknown inhomogeneities can have an arbitrary conductivity, except at the inclusion boundary (where a contrast assumption is required for visibility of the inclusion). That is, inside the inclusion (away from the inclusion boundary) there can be \emph{any} unknown anisotropic complex conductivity. To our knowledge, no other reconstruction method allows this. We also allow the background conductivity to be non-self-adjoint in certain parts of the domain, but this is arguably more of a mathematical improvement for now until the problem in Appendix~A is resolved.
	
	In Appendix~\ref{appA} we prove better monotonicity inequalities, in the sense that the bounds become zero when $A_1 = A_2$. Monotonicity-type reconstruction methods rely on localized potentials (Theorem~\ref{thm:locpot}), to control the monotonicity inequalities using e.g.\ just the electric potential $u_2$ related to coefficient $A_2$. We show that the extra terms originating from the skew-adjoint parts can be bounded solely based on $u_2$. However, these bounds are still not strong enough to avoid the extra assumptions on the background conductivity in the monotonicity method. Hence, we leave it for the inverse problems community to improve those bounds, in order to allow more general non-self-adjoint background conductivities and to enable the inner approach of the monotonicity method in this setting.
	
	\subsection{Relation to previous results}
	
	Only a few reconstruction methods of inclusions have been developed for the anisotropic Calder\'on problem, with \cite{GJZ2025} being a main recent result and a factorization method is given in \cite{Kirsch2005}. In the real isotropic (scalar-valued) setting, we will just mention a few key results \cite{GardeVogelius2024,Garde2020,Harrach13,Harrach10,Tamburrino2002} related to the monotonicity method; see references within \cite{GJZ2025,Harrach19} for additional results.
	
	In the complex isotropic setting, there is also a different type of factorization method in~\cite{Harrach09} for inclusion detection. The factorization method does not allow general perturbations, but requires the perturbation to be bounded away from zero in a definite way (either positive or negative; similarly holds for the anisotropic case in \cite{Kirsch2005}) throughout the entire inclusion. There is also a monotonicity method related to ultrasound modulated electrical impedance tomography \cite{Harrach2015}. Both results from~\cite{Harrach09,Harrach2015} rely on a clever way of weighting the ND maps with admittivity coefficients. We note that such a weighting also introduces non-uniqueness in the isotropic problem (some perturbations become invisible to measurements), and we note that this approach does not appear to generalize to the anisotropic setting.
	
	We also note a recent stability result for a certain class of complex anisotropic coefficients~\cite{Foschiatti2025}. This is for an unknown piecewise affine scalar-valued function (defined on a layered medium) multiplied by a fixed known anisotropic coefficient.
	
	\subsection{Organization of the article}
	
	Section~\ref{sec:forward} focuses on the forward conductivity problem and establishes notation related to function spaces and the sets of coefficients. The nonlinear reconstruction method (Theorem~\ref{thm:mainnonlinear}) and the linearized reconstruction method (Theorem~\ref{thm:mainlinear}) are given in Sections~\ref{sec:recnonlin} and~\ref{sec:rec}, and with their simplified versions in case $A_0$ is self-adjoint (Corollary~\ref{coro}) given in Section~\ref{sec:SA}. The reconstruction method using extreme test-operators (Theorem~\ref{thm:mainextreme}) is given in Section~\ref{sec:recextreme}. Section~\ref{sec:tools} collects a number of tools from \cite{GJZ2025} used for proving the main results in Sections~\ref{sec:proofmainnonlinear}--\ref{sec:proofmainextreme}. Appendix~\ref{appA} considers progress towards improved monotonicity inequalities.
	
	\subsection{Remarks on notation} \label{sec:notation}
	
	All considered vector spaces are complex; this includes the classes of conductivity coefficients. We use the convention that inner products on complex Hilbert spaces are linear in the first entry and anti-linear in the second. We denote the Euclidean inner product as $z_1\cdot\overline{z_2}$ for $z_1,z_2\in \C^d$, in particular the ``dot'' is bilinear. The Euclidean norm is denoted~$\abs{\,\cdot\,}$.
	
	For self-adjoint operators $A,B\in\mathscr{L}(H)$ for a Hilbert space $H$, we write $A\geq B$ when $A-B$ is positive semidefinite; this is the Loewner ordering of such operators. Moreover, for $A$ being a positive definite operator, $A^{1/2}$ denotes its unique positive definite square root.
	
	\section{The forward problem} \label{sec:forward}
	
	This section will be close to \cite[section 2]{GJZ2025}, and the claims given in this section are verified there. Let $\Omega$ be a bounded Lipschitz domain in $\R^d$, for $d\in\N\setminus\{1\}$, with connected complement. For general bounded non-self-adjoint matrix-valued functions $A\in L^\infty(\Omega)^{d\times d}$, we equip the space with the norm
	\begin{equation*}
		\norm{A}_{\mathcal{H}(\Omega)} = \sup_{x\in\Omega}\norm{A(x)}_2,
	\end{equation*}
	for which $\norm{\,\cdot\,}_2$ is the Euclidean operator norm (spectral norm for matrices). Note that in \cite{GJZ2025} we denoted this norm as $\norm{\,\cdot\,}_*$; we changed the notation to $\norm{\,\cdot\,}_{\mathcal{H}(\Omega)}$ as we will also need to consider the norm of a coefficient restricted to a subset $V\subset\Omega$, which becomes $\norm{\,\cdot\,}_{\mathcal{H}(V)}$.
	
	We consistently decompose $A$ into a \emph{self-adjoint part} $\AR$ and a \emph{skew-adjoint part} $\I\AI$, where $\AR$ and $\AI$ are self-adjoint as matrices:
	\begin{equation*}
		A = \AR + \I\AI
	\end{equation*}
	with
	\begin{equation*}
		\AR = \frac{1}{2}\bigl(A + A^*\bigr) \quad \text{and} \quad \AI = \frac{1}{2\I}\bigl(A - A^*\bigr).
	\end{equation*}
	In this notation, we may also refer to $\AR$ as the \emph{real part} and $\AI$ as the \emph{imaginary part}. Note that this does not relate to the actual matrix-entries, as these can be complex in both $\AR$ and $\AI$, but rather it relates to the associated quadratic forms, which for $\xi\in\C^d$ satisfy:
	\begin{equation*}
		\redel \bigl(A\xi\cdot\overline{\xi}\bigr) = \AR\xi\cdot\overline{\xi} \quad \text{and} \quad \imdel \bigl(A\xi\cdot\overline{\xi}\bigr) = \AI\xi\cdot\overline{\xi}.
	\end{equation*}
	
	The general class of conductivity coefficients we will consider are given by
	\begin{equation*}
		\mathcal{H}(\Omega) = \{\, A\in L^\infty(\Omega)^{d\times d} \mid \exists c>0 \colon \AR \geq c I \text{ in the Loewner order in $\Omega$} \,\}.
	\end{equation*}
	In case the conductivity coefficient is self-adjoint, it belongs to
	\begin{equation*}
		\Hsa(\Omega) = \{\, A\in\mathcal{H}(\Omega) \mid \AI = 0 \,\}.
	\end{equation*}
	
	Consider the partial data anisotropic conductivity problem for $A\in\mathcal{H}(\Omega)$: 
	\begin{equation} \label{eq:condeq}
		-\nabla\cdot(A\nabla u) = 0 \text{ in } \Omega, \qquad 	\nu\cdot(A\nabla u) = \begin{dcases}
			f & \text{on } \Gamma, \\
			0 & \text{on } \partial\Omega\setminus\Gamma.
		\end{dcases}
	\end{equation}
	Here $\nu$ is the outer unit normal to $\Omega$, $\Gamma\subseteq \partial \Omega$ is a non-empty open boundary piece, and the mean-free current density $f$ belongs to
	\begin{equation*}
		L^2_\diamond(\Gamma) = \{\, f\in L^2(\Gamma) \mid \inner{f,1} = 0 \,\}.
	\end{equation*}
	We denote by $\inner{\,\cdot\,,\,\cdot\,}$ the standard inner product on $L^2(\Gamma)$, and $\norm{\,\cdot\,}$ is the associated norm. We likewise denote
	\begin{equation*}
		H^1_\diamond(\Omega) = \{\, u\in H^1(\Omega) \mid \inner{u|_\Gamma,1} = 0 \,\}.
	\end{equation*}
	From the Lax--Milgram lemma there exists a unique weak solution $u = u_f^A \in H^1_\diamond(\Omega)$ to \eqref{eq:condeq}. We will occasionally use the notation $u_f^A$ to indicate which coefficient $A$ and Neumann condition $f$ that is considered.
	
	We let $\Lambda(A) \in \mathscr{L}(L^2_\diamond(\Gamma))$ denote the associated local ND map,
	\begin{equation*}
		\Lambda(A)f = u_f^A|_{\Gamma},
	\end{equation*}
	and the forward problem is the nonlinear mapping $\Lambda \colon \mathcal{H}(\Omega) \to \mathscr{L}(L^2_\diamond(\Gamma))$. It holds in general that 
	\begin{equation}
		\Lambda(A^*) = \Lambda(A)^*, \label{eq:NDadjoint}
	\end{equation}
	and for any $A_1,A_2\in \mathcal{H}(\Omega)$ and $f,g\in L^2_\diamond(\Gamma)$, we have
	\begin{equation} \label{eq:lambdaweak}
		\inner{f,\Lambda(A_1)g} = \int_\Omega A_2\nabla u_f^{A_2}\cdot \overline{\nabla u_g^{A_1}}\,\di x.
	\end{equation}
	We also write
	\begin{equation*}
		\LambdaR(A) = \frac{1}{2}\bigl(\Lambda(A) + \Lambda(A)^*\bigr) \quad \text{and} \quad \LambdaI(A) = \frac{1}{2\I}\bigl(\Lambda(A) - \Lambda(A)^*\bigr),
	\end{equation*}
	where $\LambdaR(A)$ and $\LambdaI(A)$ are self-adjoint operators. In terms of the quadratic forms we thus have
	\begin{align}
		\inner{f,\LambdaR(A)f} &= \redel\inner{f,\Lambda(A)f} = \int_\Omega \AR\nabla u_f^A\cdot\overline{\nabla u_f^A} \,\di x = \int_\Omega \abs{(\AR)^{1/2}\nabla u_f^A}^2\,\di x, \label{eq:lambdaR}\\
		\inner{f,\LambdaI(A)f} &= -\imdel\inner{f,\Lambda(A)f} = -\int_\Omega \AI\nabla u_f^A\cdot\overline{\nabla u_f^A} \,\di x. \label{eq:lambdaI}
	\end{align}
	
	\section{Nonlinear reconstruction of non-self-adjoint inclusions} \label{sec:recnonlin}
	
	We will assume that a ``background conductivity'' $A_0\in\mathcal{H}(\Omega)$ is known. For an unknown $A_D\in \mathcal{H}(\Omega)$, we call 
	\begin{equation*}
		D = \suppm(A_D - A_0)
	\end{equation*}
	the \emph{inclusions}, noting that $D$ may have several connected components. Our reconstruction method relates to the so-called outer shape of $D$.
	\begin{definition} \label{def:outershape}
		The \emph{outer shape} $D^\bullet$ of $D$ is the smallest closed set with connected complement such that~$D\subseteq D^\bullet$. 
	\end{definition}
	We will denote 
	\begin{equation*}
		M = \suppm(\AI_0).
	\end{equation*}
	We define the class of admissible test-inclusions $\mathcal{A}$ as the sets $C$ that satisfy:
	\begin{enumerate}[(i)]
		\item $C\Subset \Omega$ is the closure of an open set and has connected complement.
		\item $\partial C\cap M = \emptyset$.
	\end{enumerate}
	Since $M$ is known, this will inform on how to pick the test-inclusions. Note also that certain connected components of $M$ can be contained inside $C$. Moreover, the test-inclusions can consist of several connected components.
	
	\begin{assumption}  \label{assump:recon} \needspace{2\baselineskip} {}\
		\begin{enumerate}[\rm(i)]
			\item Assume there are \emph{known} bounds on $A_D$, i.e.\ scalars $0 < \alpha \leq \beta$ and $\eta\geq 0$ such that
			\begin{equation*}
				\alpha I \leq \AR_D \leq \beta I \quad \text{and} \quad \alpha I \leq \AR_0 \leq \beta I \quad
			\end{equation*}
			in the Loewner order in $\Omega$, and
			\begin{equation*}
				\max\{\norm{\AI_D}_{\mathcal{H}(\Omega)}, \norm{\AI_0}_{\mathcal{H}(\Omega)}\} \leq \eta.
			\end{equation*}
			\item Assume that $D\Subset \Omega$ and is the closure of an open set.
			\item There exists a connected component $S$ of $\Omega\setminus(D\cup M)$, such that $\partial D^\bullet \subset\partial S$ and such that $\partial S$ contains a non-empty relatively open subset of $\Gamma$.
			\item For every $x\in\partial D^\bullet$ and every open neighbourhood $W$ of $x$, assume there exists a relatively open connected set $V\subset D^\bullet\cap W$ that intersects $\partial D^\bullet$, satisfying either of two options:
			\begin{enumerate}[\rm(a)]
				\item $\AR_D - \AR_0 \geq \tau^+ I$ in $V$, and there exist an open ball $B\subset V$ and $c>0$ such that $\AR_D - \AR_0 \geq (\tau^+ + c)I$ in $B$, where
				\begin{equation*}
					\tau^+ \geq \tfrac{\beta^2}{\alpha^3}\norm{\AI_0}_{\mathcal{H}(V)}^2.
				\end{equation*}
				\item $\AR_D - \AR_0 \leq -\tau^- I$ in $V$, and there exist an open ball $B\subset V$ and $c>0$ such that $\AR_D - \AR_0 \leq -(\tau^- + c)I$ in $B$, where
				\begin{equation*}
					\tau^- \geq \tfrac{1}{\alpha}\norm{\AI_D}_{\mathcal{H}(V)}^2.
				\end{equation*}
			\end{enumerate}
			\item Assume that $\AR_0$ satisfies the UCP.
		\end{enumerate}
	\end{assumption}
	
	\begin{remark}[regarding Assumption~\ref{assump:recon}] \
		\begin{itemize}
			\item Assumption (iii) shows that all of the outer boundary of $D$ can be reached from the measurement boundary via a set on which $A_0$ is self-adjoint. This is related to the non-optimality of the monotonicity inequalities for non-self-adjoint coefficients (see e.g.~Appendix~\ref{appA}), and is a technical assumption to allow the use of localized potentials (Theorem~\ref{thm:locpot}) in the proofs.
			\item Assumption (iv) is a technical assumption about the visibility of the inclusions: When entering an inclusion, the self-adjoint part should have either a jump condition or (in case the skew-adjoint part vanishes near the inclusion boundary) could have a continuous deviation, where \emph{locally} $\AR_D - \AR_0$ becomes either positive/negative semidefinite and there is a small ball on which there is uniform positive/negative definiteness. 
			
			Note that in the assumption, at a given point $x$ on the inclusion boundary, the set $V$ does not have to include $x$. Meaning that to one side of $x$ there could be positive definiteness, while to the other side there could be negative definiteness. Hence, near different parts of the inclusion boundary we end in either case (a) or case (b) of the assumption.
			\item In the isotropic case, $\beta^2/\alpha^3$ can be replaced by $\beta/\alpha^2$ in part (a) of assumption (iv). This is due to Remark~\ref{remark:isotropicbnd} and the use of Proposition~\ref{prop:bnd} in the proofs of Theorems~\ref{thm:mainnonlinear} and~\ref{thm:mainlinear}.
		\end{itemize}
	\end{remark}
	
	In Assumption~\ref{assump:recon}, part~(iii) ensures that we can use localization away from the support of $\AI_0$, up to (but not including) an interior neighbourhood (inside $D$) of $\partial D^\bullet$, which will be crucial. 
	
	For a measurable set $C$, we define the self-adjoint test-coefficients
	\begin{equation*}
		A^-_{C} = \begin{dcases}
			\alpha I & \text{in } C \\
			\AR_0 & \text{in } \Omega\setminus C
		\end{dcases}
		\quad 
		\text{and}
		\quad
		A^+_{C} = \begin{dcases}
			(\beta+\tfrac{\eta^2}{\alpha}) I & \text{in } C \\
			\AR_0 & \text{in } \Omega\setminus C,
		\end{dcases}
	\end{equation*}
	and the corresponding test-operators
	\begin{equation*}
		\Lambda^-_C = \Lambda(A^-_C) \quad \text{and}\quad \Lambda^+_C = \Lambda(A^+_C).
	\end{equation*}
	We also define
	\begin{equation*}
		\inner{f, \DLambda_{M\setminus C}^+ f} = -\int_{M\setminus C} \AI_0(\AR_0)^{-1}\AI_0\nabla u_f^{A_C^+}\cdot\overline{\nabla u_f^{A_C^+}}\,\di x,
	\end{equation*}
	which is the Fr\'echet derivative of $\Lambda$ at $A_C^+$ and in the direction $\AI_0(\AR_0)^{-1}\AI_0\chi_{M\setminus C}$, where $\chi_{M\setminus C}$ is a characteristic function on the set $M\setminus C$, cf.~\cite{Garde2022c}.
	
	\begin{theorem} \label{thm:mainnonlinear}
		Under Assumption~\ref{assump:recon}(i), for any measurable $C\subseteq \overline{\Omega}$ we have
		\begin{equation*}
			D \subseteq C \quad \text{implies} \quad \Lambda^-_{C} \geq \LambdaR(A_D) \geq \Lambda^+_{C}+\DLambda_{M\setminus C}^+.
		\end{equation*}
		Under all of Assumption~\ref{assump:recon}, for any $C\in\mathcal{A}$ we have
		\begin{equation*}
			\Lambda^-_{C} \geq \LambdaR(A_D) \geq \Lambda^+_{C}+\DLambda_{M\setminus C}^+ \quad \text{implies} \quad D \subseteq C.
		\end{equation*}
	\end{theorem}
	\begin{proof}
		The proof is given in Section~\ref{sec:proofmainnonlinear}.
	\end{proof}
	
	\section{Linearized reconstruction of non-self-adjoint inclusions} \label{sec:rec}
	
	We now consider a linearized version of the monotonicity method. We define the following test-operators, via quadratic forms,
	\begin{align*}
		\inner{f,\DLambda_C^+ f} &= -\int_C \bigl[\bigl(\beta + \tfrac{\eta^2}{\alpha}\bigr)I-\AR_0\bigr]\nabla u_f^{\AR_0}\cdot\overline{\nabla u_f^{\AR_0}}\,\di x, \\
		\inner{f,\DLambda_C^- f} &= -\int_C \bigl[\AR_0 - \tfrac{\beta^2}{\alpha}I\bigr]\nabla u_f^{\AR_0}\cdot\overline{\nabla u_f^{\AR_0}}\,\di x, \\
		\inner{f, \DLambda_{M\setminus C} f} &= -\int_{M\setminus C} \AI_0(\AR_0)^{-1}\AI_0\nabla u_f^{\AR_0}\cdot\overline{\nabla u_f^{\AR_0}}\,\di x.
	\end{align*}
	The above test-operators are in fact Fr\'echet derivatives of $\Lambda$ at $\AR_0$ and in particular directions, cf.~\cite{Garde2022c}. Note below that we have $\Lambda(\AR_0) = \LambdaR(\AR_0)$ because of \eqref{eq:NDadjoint}.
	
	\begin{theorem} \label{thm:mainlinear}
		Under Assumption~\ref{assump:recon}(i), for any measurable $C\subseteq \overline{\Omega}$ we have
		\begin{equation*}
			D \subseteq C \quad \text{implies} \quad \DLambda^-_{C} \geq \LambdaR(A_D) - \Lambda(\AR_0) \geq \DLambda^+_{C}+\DLambda_{M\setminus C}.
		\end{equation*}
		Under all of Assumption~\ref{assump:recon}, for any $C\in\mathcal{A}$ we have
		\begin{equation*}
			 \DLambda^-_{C} \geq \LambdaR(A_D) - \Lambda(\AR_0) \geq \DLambda^+_{C}+\DLambda_{M\setminus C} \quad \text{implies} \quad D \subseteq C.
		\end{equation*}
	\end{theorem}
	\begin{proof}
		The proof is given in Section~\ref{sec:proofmainlinear}.
	\end{proof}
	
	\begin{remark}
		Under the given assumptions, Theorem~\ref{thm:mainnonlinear} gives
		\begin{equation*}
			D^\bullet = \cap\, \{\, C\in\mathcal{A} \mid \Lambda^-_{C} \geq \LambdaR(A_D) \geq \Lambda^+_{C}+\DLambda_{M\setminus C}^+ \,\},
		\end{equation*}
		while Theorem~\ref{thm:mainlinear} gives
		\begin{equation*}
			D^\bullet = \cap\, \{\, C\in\mathcal{A} \mid \DLambda^-_{C} \geq \LambdaR(A_D) - \Lambda(\AR_0) \geq \DLambda^+_{C}+\DLambda_{M\setminus C} \,\}.
		\end{equation*}
		Moreover, from the proofs we also conclude:
		\begin{itemize}
			\item If in Assumption~\ref{assump:recon}(iv) we only arrive at case (a), we only need to check the inequalities $\LambdaR(A_D) \geq \Lambda^+_{C}+\DLambda_{M\setminus C}^+$ or $\LambdaR(A_D) - \Lambda(\AR_0) \geq \DLambda^+_{C}+\DLambda_{M\setminus C}$.
			\item If in Assumption~\ref{assump:recon}(iv) we only arrive at case (b), we only need to check the inequalities $\Lambda^-_{C} \geq \LambdaR(A_D)$ or $\DLambda^-_{C} \geq \LambdaR(A_D) - \Lambda(\AR_0)$.
		\end{itemize} 
	\end{remark}
	
	\section{When the background conductivity is self-adjoint} \label{sec:SA}
	
	We will here give the corresponding results to Theorems~\ref{thm:mainnonlinear} and~\ref{thm:mainlinear} in the simplified setting when $A_0$ is self-adjoint, that is $A_0\in\Hsa(\Omega)$. We will also assume that $\AI_D = 0$ in a neighbourhood of $\partial D^\bullet$ to further simplify the definiteness assumptions. In this case the class of admissible test-inclusions is
	\begin{equation*}
		\mathcal{A} = \{\, C \Subset \Omega \mid C \text{ is the closure of an open set and has connected complement} \,\}.
	\end{equation*}
	\begin{assumption}  \label{assump:recon2} \needspace{2\baselineskip} {}\
		\begin{enumerate}[\rm(i)]
			\item Assume there are \emph{known} bounds on $A_D$, i.e.\ scalars $0 < \alpha \leq \beta$ and $\eta\geq 0$ such that
			\begin{equation*}
				\alpha I \leq \AR_D \leq \beta I \quad \text{and} \quad \alpha I \leq \AR_0 \leq \beta I \quad
			\end{equation*}
			in the Loewner order in $\Omega$, and
			\begin{equation*}
				\norm{\AI_D}_{\mathcal{H}(\Omega)} \leq \eta.
			\end{equation*}
			\item Assume that $D\Subset \Omega$ and is the closure of an open set.
			\item Assume $\AI_D = 0$ in an open neighbourhood of $\partial D^\bullet$.
			\item For every $x\in\partial D^\bullet$ and every open neighbourhood $W$ of $x$, assume there exists a relatively open connected set $V\subset D^\bullet\cap W$ that intersects $\partial D^\bullet$, satisfying either of two options:
			\begin{enumerate}[\rm(a)]
				\item $\AR_D - A_0$ is positive semidefinite in $V$, and there exists an open ball $B\subset V$ on which $\AR_D - A_0$ is uniformly positive definite.
				\item $\AR_D - A_0$ is negative semidefinite in $V$, and there exists an open ball $B\subset V$ on which $\AR_D - A_0$ is uniformly negative definite.
			\end{enumerate}
			\item Assume that $A_0\in\Hsa(\Omega)$ satisfies the UCP.
		\end{enumerate}
	\end{assumption}	
		
	Under these assumptions, Theorem~\ref{thm:mainnonlinear} and Theorem~\ref{thm:mainlinear} reduce to the following results.
	
	\begin{corollary} \label{coro}
		Under Assumption~\ref{assump:recon2}(i), for any measurable $C\subseteq \overline{\Omega}$ we have
		\begin{equation*}
			D \subseteq C \quad \text{implies both} \quad \Lambda^-_{C} \geq \LambdaR(A_D) \geq \Lambda^+_{C} \quad \text{and} \quad \DLambda^-_{C} \geq \LambdaR(A_D) - \Lambda(A_0) \geq \DLambda^+_{C}.
		\end{equation*}
		Under all of Assumption~\ref{assump:recon2}, for any $C\in\mathcal{A}$ we have
		\begin{equation*}
			\text{both} \quad \Lambda^-_{C} \geq \LambdaR(A_D) \geq \Lambda^+_{C} \quad \text{and} \quad \DLambda^-_{C} \geq \LambdaR(A_D) - \Lambda(A_0) \geq \DLambda^+_{C} \quad \text{imply} \quad D \subseteq C.
		\end{equation*}
	\end{corollary}
	
	This resembles the results from \cite[theorems 3.6 and 3.7]{GJZ2025}, with the exception that $\Lambda^+_C$ and $\DLambda^+_{C}$ take the bound on $\AI_D$ into account, and gives precisely \cite[theorems 3.6 and 3.7]{GJZ2025} in case $A_D$ is self-adjoint.
	
	\section{Reconstruction using extreme test-operators} \label{sec:recextreme}
	
	We will now consider the use of \emph{extreme} test-operators for the monotonicity method with non-self-adjoint inclusions, in the style of \cite[theorem~12.2]{GJZ2025}. However, a notable difference is that we will not include extreme inclusions in $A_D$. Hence, we will assume that $A_D\in\mathcal{H}(\Omega)$, so the main novelty is that we will not require bounds on $\AR_D$ or $\AI_D$ (unlike in Corollary~\ref{coro}). Similar to Section~\ref{sec:SA}, we assume that $A_0\in \Hsa(\Omega)$.
	
	We define the class of admissible test-inclusions, which for the extreme test-operators also have Lipschitz regular boundary:
	\begin{align*}
		\widehat{\mathcal{A}} &= \{\, C \Subset \Omega \mid C \text{ is the closure of an open set,}  \\
		&\hphantom{{}= \{C \subset \Omega \mid{}\,}\text{has connected complement,} \\
		&\hphantom{{}= \{C \subset \Omega \mid{}\,}\text{and has Lipschitz boundary } \partial C \,\}.
	\end{align*}
	
	\begin{assumption}  \label{assump:reconextreme} \needspace{2\baselineskip} {}\
		\begin{enumerate}[\rm(i)]
			\item Assume that $D\Subset \Omega$ and is the closure of an open set, and that $\partial D^\bullet$ is Lipschitz regular.
			\item Assume $\AI_D = 0$ in an open neighbourhood of $\partial D^\bullet$.
			\item For every $x\in\partial D^\bullet$ and every open neighborhood $W$ of $x$, assume there exists a relatively open connected set $V\subset D^\bullet\cap W$ that intersects $\partial D^\bullet$, satisfying either of two options:
			\begin{enumerate}[\rm(a)]
				\item $\AR_D - A_0$ is positive semidefinite in $V$, and there exists an open ball $B\subset V$ on which $\AR_D - A_0$ is uniformly positive definite.
				\item $\AR_D - A_0$ is negative semidefinite in $V$, and there exists an open ball $B\subset V$ on which $\AR_D - A_0$ is uniformly negative definite.
			\end{enumerate}
			\item Assume that $A_0\in\Hsa(\Omega)$ satisfies the UCP.
		\end{enumerate}
	\end{assumption}
	
	We will use the shorthand notation $\Lambda_C^{\emptyset} = \Lambda(A_C^{\emptyset})$ and $\Lambda_{\emptyset}^C = \Lambda(A_{\emptyset}^C)$ for the extreme coefficients
	\begin{equation*}
		A_C^{\emptyset} = \begin{dcases}
			0 I & \text{in } C \\
			A_0 & \text{in } \Omega\setminus C
		\end{dcases}
		\quad 
		\text{and}
		\quad
		A_{\emptyset}^C = \begin{dcases}
			\infty I & \text{in } C \\
			A_0 & \text{in } \Omega\setminus C.
		\end{dcases}
	\end{equation*}
	Here $A_C^\emptyset$ corresponds to having a perfectly insulating part in $C$ and $A_\emptyset^C$ corresponds to having a perfectly conducting part in $C$. For the precise definition of the PDE problem associated with extreme coefficients, we refer to \cite[section 11]{GJZ2025}. This leads to the following result.
		
	\begin{theorem} \label{thm:mainextreme}
		For any $C\in\widehat{\mathcal{A}}$ we have
		\begin{equation*}
			D \subseteq C \quad \text{implies} \quad \Lambda_C^{\emptyset} \geq \LambdaR(A_D) \geq \Lambda_{\emptyset}^C.
		\end{equation*}
		Under Assumption~\ref{assump:reconextreme}, for any $C\in\widehat{\mathcal{A}}$ we have
		\begin{equation*}
			\Lambda_C^{\emptyset} \geq \LambdaR(A_D) \geq \Lambda_{\emptyset}^C \quad \text{implies} \quad D \subseteq C.
		\end{equation*}
	\end{theorem}
	\begin{proof}
		The proof is given in Section~\ref{sec:proofmainextreme}.
	\end{proof}
	
	\begin{remark}
		Under Assumptions~\ref{assump:reconextreme}, Theorem~\ref{thm:mainextreme} gives
		\begin{equation*}
			D^\bullet = \cap\, \{\, C\in\widehat{\mathcal{A}} \mid \Lambda_C^{\emptyset} \geq \LambdaR(A_D) \geq \Lambda_{\emptyset}^C \,\}.
		\end{equation*}
		Moreover, from the proof we also conclude:
		\begin{itemize}
			\item If in Assumption~\ref{assump:reconextreme}(iii) there is only positive definiteness near $\partial D^\bullet$, we only need to check the inequality $\LambdaR(A_D) \geq \Lambda_{\emptyset}^C$.
			\item If in Assumption~\ref{assump:reconextreme}(iii) there is only negative definiteness near $\partial D^\bullet$, we only need to check the inequality $\Lambda_C^{\emptyset} \geq \LambdaR(A_D)$.
		\end{itemize} 
	\end{remark}
	
	\section{Tools for the monotonicity method} \label{sec:tools}
	
	This section collects some of the main tools on monotonicity and localization from \cite{GJZ2025}, needed for proving the main results.
	
	Note that 
	\begin{equation*}
		-2[\AR_2(\AR_1)^{-1}\AI_2]\upI = \I\bigl[\AR_2(\AR_1)^{-1}\AI_2 - \AI_2(\AR_1)^{-1}\AR_2\bigr].
	\end{equation*}
	This leads to the following monotonicity inequalities.
	
	\begin{theorem}[Theorem 5.1 with $\kappa = 1$ in \cite{GJZ2025}] \label{thm:generalmono}
		Let $A_1,A_2\in\mathcal{H}(\Omega)$ and $f\in L^2_\diamond(\Gamma)$. For $j\in\{1,2\}$ we denote $\Lambda_j = \Lambda(A_j)$ and $u_j = u_f^{A_j}$. Then
		\begin{align*}
			\inner{f,(\LambdaR_1-\LambdaR_2) f} &\geq \int_\Omega \bigl[(\AR_2-\AR_1) - \AI_1(\AR_1)^{-1}\AI_1 \bigr]\nabla u_2\cdot \overline{\nabla u_2}\,\di x,  \\
			\inner{f,(\LambdaR_1-\LambdaR_2) f} &\leq \int_\Omega \bigl[ \AR_2(\AR_1)^{-1}(\AR_2-\AR_1) + \AI_2(\AR_1)^{-1}\AI_2 - 2[\AR_2(\AR_1)^{-1}\AI_2]\upI \bigr]\nabla u_2\cdot\overline{\nabla u_2}\,\di x.
		\end{align*}
	\end{theorem}
	
	\begin{remark} \label{remark:lastmonoterm}
		The term $[\AR_2(\AR_1)^{-1}\AI_2]\upI$ vanishes in the isotropic case by commutativity. In the anisotropic case it also vanishes in the parts of the domain where $\AR_1 = \AR_2$ or where $\AI_2 = 0$.
	\end{remark}
	
	\begin{proposition}[Proposition 5.6 in \cite{GJZ2025}] \label{prop:bnd}
		Let $c\geq 0$ and $V \subseteq \Omega$. Let $A_1,A_2\in\Hsa(\Omega)$ with $A_1 \leq \beta I$ and $A_2 \geq \alpha I$  in $V$ for $\alpha,\beta>0$. 
		\begin{enumerate}[\rm(i)]
			\item If $A_2-A_1 \geq cI$ in $V$, then $A_2A_1^{-1}(A_2-A_1) \geq cI$ in $V$.
			\item If $A_2-A_1 \leq -cI$ in $V$, then $A_2A_1^{-1}(A_2-A_1) \leq -c(\tfrac{\alpha}{\beta})^2I$ in $V$.
		\end{enumerate}
	\end{proposition}
	
	\begin{remark} \label{remark:isotropicbnd}
		In the isotropic case, $(\alpha/\beta)^2$ can be replaced by $\alpha/\beta$ in Proposition~\ref{prop:bnd}(ii).
	\end{remark}
	
	\begin{theorem}[Theorem 6.3 in \cite{GJZ2025}] \label{thm:locpot}
		Let $U\subset \overline{\Omega}$ be a relatively open connected set that intersects $\Gamma$ and let $B\subset U$ be a non-empty open set. Assume $A\in \mathcal{H}(\Omega)$ where both $A$ and $A^*$ satisfy the UCP in $U$. Then there are sequences $(f_j)$ in $L^2_\diamond(\Gamma)$ and $(u_j)$ in $H^1_\diamond(\Omega)$, with $u_j = u_{f_j}^A$, such that
		\begin{equation*}
			\lim_{j\to\infty} \int_B \abs{\nabla u_j}^2\,\di x = \infty \qquad \text{and} \qquad \lim_{j\to\infty}\int_{\Omega\setminus U} \abs{\nabla u_j}^2\,\di x = 0. 
		\end{equation*}
	\end{theorem}
	
	\section{Proof of Theorem \ref{thm:mainnonlinear}} \label{sec:proofmainnonlinear}
	
	First note that
	\begin{equation} \label{eq:matbnd0}
		\AI_D(\AR_D)^{-1}\AI_D\xi\cdot\overline{\xi} = (\AR_D)^{-1}\AI_D\xi\cdot\overline{\AI_D\xi} \leq \alpha^{-1}\abs{\AI_D\xi}^2\leq \tfrac{\eta^2}{\alpha}\abs{\xi}^2.
	\end{equation}
	Let $D\subseteq C$. Let $u_D = u_f^{A_D}$ and $u_+ = u_f^{A_C^+}$. From Theorem~\ref{thm:generalmono} and \eqref{eq:matbnd0} we have
	\begin{align*}
		\inner{f,\bigl[ \LambdaR(A_D) - \Lambda_C^+ - \DLambda_{M\setminus C}^+ \bigr]f} &\geq \int_\Omega \bigl[ A_C^+ - \AR_D - \AI_D(\AR_D)^{-1}\AI_D \bigr] \nabla u_+\cdot\overline{\nabla u_+}\,\di x \\
		&\phantom{={}} + \int_{M\setminus C} \AI_0(\AR_0)^{-1}\AI_0\nabla u_+ \cdot\overline{\nabla u_+}\,\di x \\
		&= \int_C \bigl[ (\beta + \tfrac{\eta^2}{\alpha})I - \AR_D - \AI_D(\AR_D)^{-1}\AI_D \bigr]\nabla u_+ \cdot\overline{\nabla u_+}\,\di x \geq 0.
	\end{align*}
	For the other inequality, we have from Theorem~\ref{thm:generalmono}:
	\begin{align*}
		\inner{f,\bigl[ \Lambda_C^- - \LambdaR(A_D) \bigr]f} &\geq \int_\Omega (\AR_D - A_C^-)\nabla u_D\cdot\overline{\nabla u_D}\,\di x \\
		&= \int_C (\AR_D - \alpha I)\nabla u_D\cdot\overline{\nabla u_D}\,\di x \geq 0.
	\end{align*}
	
	Now assume that $D \not\subseteq C$, i.e.~that $D^\bullet\not\subseteq C$. Since both sets are closures of open sets and have connected complements, there is a relatively open set that connects $D^\bullet\setminus C$ with $\Gamma$ that we may use for localization. Specifically, there exist a relatively open connected set $U\subset (D\cup S)\setminus C$ that intersects $\Gamma$ and an open ball $B\subset U \cap D$. Moreover, Assumption~\ref{assump:recon}(iv) ensures that we can pick $B$ and $U$ such that we arrive at one of two cases:
	\begin{itemize}
		\item Case 1 corresponds to part (a) of Assumption~\ref{assump:recon}(iv) where $V = U\cap D$.
		\item Case 2 corresponds to part (b) of Assumption~\ref{assump:recon}(iv) where $V = U\cap D$.
	\end{itemize}
	In the first case we will prove that $\LambdaR(A_D) \not\geq \Lambda^+_{C} + \DLambda_{M\setminus C}^+$, and in the second case we will prove that $\Lambda^-_{C} \not\geq \LambdaR(A_D)$. Together the two cases prove the contrapositive formulation of the assertion 
	\begin{equation*}
		\Lambda^-_{C} \geq \LambdaR(A_D) \geq \Lambda^+_{C} + \DLambda_{M\setminus C}^+ \quad \text{\emph{implies}} \quad D \subseteq C.
	\end{equation*}
	
	\subsection*{Case 1}
	
	Since $C\subset \Omega\setminus U$, from Theorem~\ref{thm:locpot} there is a sequence $(f_j)$ in $L^2_\diamond(\Gamma)$ such that $u_j = u_{f_j}^{A_C^+}$ satisfy
	\begin{equation*}
		\lim_{j\to\infty} \int_B \abs{\nabla u_j}^2\,\di x = \infty \qquad \text{and} \qquad \lim_{j\to\infty}\int_{\Omega\setminus U} \abs{\nabla u_j}^2\,\di x = 0.
	\end{equation*}
	From Theorem~\ref{thm:generalmono} we have
	\begin{align*}
		\inner{f_j,\bigl[ \LambdaR(A_D) - \Lambda_C^+ - \DLambda_{M\setminus C}^+ \bigr]f_j} &\leq \int_\Omega A_C^+(\AR_D)^{-1}(A_C^+ - \AR_D)\nabla u_j\cdot\overline{\nabla u_j}\,\di x \\
		&\phantom{{}\leq}+ \int_{M\setminus C}\AI_0(\AR_0)^{-1}\AI_0\nabla u_j\cdot\overline{\nabla u_j}\,\di x.
	\end{align*}
	Let 
	\begin{align*}
		\textup{I}^1_j &= \int_{\Omega\setminus U} A_C^+(\AR_D)^{-1}(A_C^+ - \AR_D)\nabla u_j\cdot\overline{\nabla u_j}\,\di x, \\
		\textup{I}^2_j &= \int_{M\setminus(U\cup C)} \AI_0(\AR_0)^{-1}\AI_0\nabla u_j\cdot\overline{\nabla u_j}\,\di x, 
	\end{align*}
	then we have $\lim_{j\to\infty} \textup{I}_j^k = 0$ for $k\in\{1,2\}$ as these integrals are evaluated inside $\Omega\setminus U$. Since $\AI_0 = 0$ in $U\setminus D$ (from the choice of $U$), we have
	\begin{equation*}
		\inner{f_j,\bigl[ \LambdaR(A_D) - \Lambda_C^+ - \DLambda_{M\setminus C}^+ \bigr]f_j} \leq \int_{U\cap D} \bigl[\AR_0(\AR_D)^{-1}(\AR_0 - \AR_D) + \AI_0(\AR_0)^{-1}\AI_0 \bigr]\nabla u_j\cdot\overline{\nabla u_j}\,\di x + \textup{I}^1_j + \textup{I}^2_j.
	\end{equation*}
	Based on Assumption~\ref{assump:recon}(iv) part (a), since $\AR_0 - \AR_D \leq -\tau^+ I$ and $\AI_0(\AR_0)^{-1}\AI_0 \leq (\tfrac{\alpha}{\beta})^2\tau^+ I$ in $U\cap D$, and $\AR_0 - \AR_D \leq -(\tau^+ + c) I$ in~$B$, then Proposition~\ref{prop:bnd}(ii) gives
	\begin{equation*}
		\inner{f_j,\bigl[ \LambdaR(A_D) - \Lambda_C^+ - \DLambda_{M\setminus C}^+ \bigr]f_j} \leq -c(\tfrac{\alpha}{\beta})^2\int_B \abs{\nabla u_j}^2\,\di x + \textup{I}^1_j + \textup{I}^2_j \to -\infty \text{ for } j\to\infty.
	\end{equation*}
	In particular, $\LambdaR(A_D) \not\geq \Lambda^+_{C} + \DLambda_{M\setminus C}^+$.
	
	\subsection*{Case 2}
	
	Since $C\subset \Omega\setminus U$, from Theorem~\ref{thm:locpot} there is a sequence $(f_j)$ in $L^2_\diamond(\Gamma)$ such that $u_j = u_{f_j}^{A_C^-}$ satisfy
	\begin{equation*}
		\lim_{j\to\infty} \int_B \abs{\nabla u_j}^2\,\di x = \infty \qquad \text{and} \qquad \lim_{j\to\infty}\int_{\Omega\setminus U} \abs{\nabla u_j}^2\,\di x = 0.
	\end{equation*}
	Let 
	\begin{equation*}
		\textup{I}_j = \int_{\Omega\setminus U} \bigl[ \AR_D - A_C^- + \AI_D(\AR_D)^{-1}\AI_D \bigr]\nabla u_j \cdot \overline{\nabla u_j}\,\di x,
	\end{equation*}
	then we have $\lim_{j\to\infty} \textup{I}_j = 0$. From Theorem~\ref{thm:generalmono} and as $\AI_0 = 0$ in $U\setminus D$ (from the choice of $U$), we have
	\begin{equation*}
		\inner{f_j,\bigl[ \Lambda_C^- - \LambdaR(A_D) \bigr]f_j} \leq \int_{U\cap D} \bigl[\AR_D - \AR_0 + \AI_D(\AR_D)^{-1}\AI_D\bigr]\nabla u_j\cdot\overline{\nabla u_j}\,\di x + \textup{I}_j.
	\end{equation*}
	Based on Assumption~\ref{assump:recon}(iv) part (b), since $\AR_D - \AR_0 \leq -\tau^- I$ and $\AI_D(\AR_D)^{-1}\AI_D \leq \tau^- I$ in $U\cap D$, and $\AR_D - \AR_0 \leq -(\tau^- + c) I$ in~$B$, then 
	\begin{equation*}
		\inner{f_j,\bigl[ \Lambda_C^- - \LambdaR(A_D) \bigr]f_j} \leq -c\int_B \abs{\nabla u_j}^2\,\di x + \textup{I}_j \to -\infty \text{ for } j\to\infty.
	\end{equation*}
	In particular, $\Lambda_C^- \not\geq \LambdaR(A_D)$.
	
	\section{Proof of Theorem \ref{thm:mainlinear}} \label{sec:proofmainlinear}
	
	Let $D\subseteq C$ and let $u_0 = u_f^{\AR_0}$. From Theorem~\ref{thm:generalmono} and \eqref{eq:matbnd0} we have
	\begin{align*}
		\inner{f,\bigl[ \LambdaR(A_D)-\Lambda(\AR_0)-\DLambda_C^+ - \DLambda_{M\setminus C} \bigr]f} \hspace{-6cm}&\\
		&\geq \int_\Omega \bigl[ \AR_0 - \AR_D - \AI_D(\AR_D)^{-1}\AI_D \bigr]\nabla u_0\cdot\overline{\nabla u_0}\,\di x + \int_C \bigl[ \bigl(\beta+\tfrac{\eta^2}{\alpha}\bigr)I - \AR_0 \bigr]\nabla u_0\cdot\overline{\nabla u_0}\,\di x \\
		&\phantom{={}}  + \int_{M\setminus C} \AI_0(\AR_0)^{-1}\AI_0\nabla u_0\cdot\overline{\nabla u_0}\,\di x \\
		&= \int_D \bigl[ \AR_0 - \AR_D - \AI_D(\AR_D)^{-1}\AI_D \bigr]\nabla u_0\cdot\overline{\nabla u_0}\,\di x + \int_C \bigl[ \bigl(\beta+\tfrac{\eta^2}{\alpha}\bigr)I - \AR_0 \bigr]\nabla u_0\cdot\overline{\nabla u_0}\,\di x \\
		&\phantom{={}} - \int_{M\setminus D} \AI_0(\AR_0)^{-1}\AI_0\nabla u_0\cdot\overline{\nabla u_0}\,\di x + \int_{M\setminus C} \AI_0(\AR_0)^{-1}\AI_0\nabla u_0\cdot\overline{\nabla u_0}\,\di x \\
		&\geq \int_{C\setminus D} \bigl[ \bigl(\beta+\tfrac{\eta^2}{\alpha}\bigr)I - \AR_0 \bigr]\nabla u_0\cdot\overline{\nabla u_0}\,\di x - \int_{M\cap (C\setminus D)} \AI_0(\AR_0)^{-1}\AI_0\nabla u_0\cdot\overline{\nabla u_0}\,\di x \\
		&\geq \int_{M\cap (C\setminus D)} \bigl[\tfrac{\eta^2}{\alpha}I - \AI_0(\AR_0)^{-1}\AI_0 \bigr] \nabla u_0\cdot\overline{\nabla u_0}\,\di x \geq 0.
	\end{align*}
	For the other inequality, we first note that
	\begin{equation} \label{eq:matbnd}
		\AR_0(\AR_D)^{-1}\AR_0\xi\cdot\overline{\xi} = (\AR_D)^{-1}\AR_0\xi\cdot\overline{\AR_0\xi} \leq \alpha^{-1}\abs{\AR_0\xi}^2\leq \tfrac{\beta^2}{\alpha}\abs{\xi}^2.
	\end{equation}
	Hence we have from Theorem~\ref{thm:generalmono} and~\eqref{eq:matbnd}:
	\begin{align*}
		\inner{f,\bigl[ \Lambda(\AR_0)-\LambdaR(A_D)+\DLambda_C^- \bigr]f} \hspace{-4cm}&\\
		&\geq \int_\Omega \AR_0(\AR_D)^{-1}(\AR_D-\AR_0)\nabla u_0\cdot\overline{\nabla u_0}\,\di x + \int_C \bigl[\tfrac{\beta^2}{\alpha}I-\AR_0\bigr]\nabla u_0\cdot\overline{\nabla u_0}\,\di x \\
		&= \int_D \bigl[\AR_0-\AR_0(\AR_D)^{-1}\AR_0\bigr]\nabla u_0\cdot\overline{\nabla u_0}\,\di x + \int_C \bigl[\tfrac{\beta^2}{\alpha}I-\AR_0\bigr]\nabla u_0\cdot\overline{\nabla u_0}\,\di x \\
		&\geq \int_{C\setminus D} \bigl[\tfrac{\beta^2}{\alpha}I-\AR_0\bigr]\nabla u_0\cdot\overline{\nabla u_0}\,\di x \geq 0.
	\end{align*}
	
	Now assume that $D \not\subseteq C$, i.e.~that $D^\bullet\not\subseteq C$. Since both sets are closures of open sets and have connected complements, there is a relatively open set that connects $D^\bullet\setminus C$ with $\Gamma$ that we may use for localization. Specifically, there exist a relatively open connected set $U\subset (D\cup S)\setminus C$ that intersects $\Gamma$ and an open ball $B\subset U \cap D$. Moreover, Assumption~\ref{assump:recon}(iv) ensures that we can pick $B$ and $U$ such that we arrive at one of two cases:
	\begin{itemize}
		\item Case 1 corresponds to part (a) of Assumption~\ref{assump:recon}(iv) where $V = U\cap D$.
		\item Case 2 corresponds to part (b) of Assumption~\ref{assump:recon}(iv) where $V = U\cap D$.
	\end{itemize}
	In the first case we will prove that $\LambdaR(A_D) - \Lambda(\AR_0) \not\geq \DLambda^+_{C} + \DLambda_{M\setminus C}$, and in the second case we will prove that $\DLambda^-_{C} \not\geq \LambdaR(A_D) - \Lambda(\AR_0)$. Together the two cases prove the contrapositive formulation of the assertion 
	\begin{equation*}
		\DLambda^-_{C} \geq \LambdaR(A_D) - \Lambda(\AR_0) \geq \DLambda^+_{C}+\DLambda_{M\setminus C} \quad \text{\emph{implies}} \quad D \subseteq C.
	\end{equation*}
	For both cases we use the same sequence of localized potentials. From Theorem~\ref{thm:locpot} there is a sequence $(f_j)$ in $L^2_\diamond(\Gamma)$ such that $u_j = u_{f_j}^{\AR_0}$ satisfy
	\begin{equation*}
		\lim_{j\to\infty} \int_B \abs{\nabla u_j}^2\,\di x = \infty \qquad \text{and} \qquad \lim_{j\to\infty}\int_{\Omega\setminus U} \abs{\nabla u_j}^2\,\di x = 0.
	\end{equation*}
	
	\subsection*{Case 1}
	
	Let 
	\begin{align*}
		\textup{I}^1_j &= -\inner{f_j,\DLambda_C^+ f_j}, \\
		\textup{I}^2_j &= \int_{\Omega\setminus U} \AR_0(\AR_D)^{-1}(\AR_0 - \AR_D)\nabla u_j\cdot\overline{\nabla u_j}\,\di x, \\
		\textup{I}^3_j &= \int_{M\setminus(U\cup C)} \AI_0(\AR_0)^{-1}\AI_0\nabla u_j\cdot\overline{\nabla u_j}\,\di x,
	\end{align*}
	then we have $\lim_{j\to\infty} \textup{I}_j^k = 0$ for $k\in\{1,2,3\}$ as these integrals are evaluated inside $\Omega\setminus U$. From Theorem~\ref{thm:generalmono} and as $\AI_0 = 0$ in $U\setminus D$ (from the choice of $U$), we have
	\begin{align*}
		\inner{f_j,\bigl[ \LambdaR(A_D)-\Lambda(\AR_0)-\DLambda_C^+ - \DLambda_{M\setminus C} \bigr]f_j} \hspace{-6cm}&\\
		&\leq \int_\Omega \AR_0(\AR_D)^{-1}(\AR_0 - \AR_D)\nabla u_j\cdot\overline{\nabla u_j}\,\di x + \textup{I}_j^1  + \int_{M\setminus C} \AI_0(\AR_0)^{-1}\AI_0\nabla u_j\cdot\overline{\nabla u_j}\,\di x \\
		&= \int_{U\cap D} \bigl[\AR_0(\AR_D)^{-1}(\AR_0 - \AR_D) + \AI_0(\AR_0)^{-1}\AI_0 \bigr]\nabla u_j\cdot\overline{\nabla u_j}\,\di x + \sum_{k=1}^3 \textup{I}_j^k.
	\end{align*}
	Based on Assumption~\ref{assump:recon}(iv) part (a), since $\AR_0 - \AR_D \leq -\tau^+ I$ and $\AI_0(\AR_0)^{-1}\AI_0 \leq (\tfrac{\alpha}{\beta})^2\tau^+ I$ in $U\cap D$, and $\AR_0 - \AR_D \leq -(\tau^+ + c) I$ in~$B$, then Proposition~\ref{prop:bnd}(ii) gives
	\begin{equation*}
		\inner{f_j,\bigl[ \LambdaR(A_D)-\Lambda(\AR_0)-\DLambda_C^+ - \DLambda_{M\setminus C} \bigr]f_j} \leq -c(\tfrac{\alpha}{\beta})^2\int_B \abs{\nabla u_j}^2\,\di x + \sum_{k=1}^3 \textup{I}_j^k \to -\infty \text{ for } j\to\infty.
	\end{equation*}
	In particular, $\LambdaR(A_D)-\Lambda(\AR_0) \not\geq \DLambda_C^+ + \DLambda_{M\setminus C}$.
	
	\subsection*{Case 2}
	
	Let 
	\begin{align*}
		\textup{I}^1_j &= \inner{f_j,\DLambda_C^- f_j}, \\
		\textup{I}^2_j &= \int_{\Omega\setminus U} \bigl[\AR_D-\AR_0 + \AI_D(\AR_D)^{-1}\AI_D \bigr]\nabla u_j\cdot\overline{\nabla u_j}\,\di x, 
	\end{align*}
	then we have $\lim_{j\to\infty} \textup{I}_j^k = 0$ for $k\in\{1,2\}$ as these integrals are evaluated inside $\Omega\setminus U$. From Theorem~\ref{thm:generalmono} and as $\AI_0 = 0$ in $U\setminus D$ (from the choice of $U$), we have
	\begin{align*}
		\inner{f_j,\bigl[ \Lambda(\AR_0)-\LambdaR(A_D)+\DLambda_C^- \bigr]f_j} &\leq \int_\Omega \bigl[\AR_D-\AR_0 + \AI_D(\AR_D)^{-1}\AI_D \bigr]\nabla u_j\cdot\overline{\nabla u_j}\,\di x + \textup{I}_j^1 \\
		&= \int_{U\cap D} \bigl[\AR_D-\AR_0 + \AI_D(\AR_D)^{-1}\AI_D \bigr]\nabla u_j\cdot\overline{\nabla u_j}\,\di x  + \textup{I}_j^1 + \textup{I}_j^2.
	\end{align*}
	Based on Assumption~\ref{assump:recon}(iv) part (b), since $\AR_D - \AR_0 \leq -\tau^- I$ and $\AI_D(\AR_D)^{-1}\AI_D \leq \tau^- I$ in $U\cap D$, and $\AR_D - \AR_0 \leq -(\tau^- + c) I$ in~$B$, then 
	\begin{equation*}
		\inner{f_j,\bigl[ \Lambda(\AR_0)-\LambdaR(A_D)+\DLambda_C^- \bigr]f_j} \leq -c\int_B \abs{\nabla u_j}^2\,\di x + \textup{I}_j^1 + \textup{I}_j^2 \to -\infty \text{ for } j\to\infty.
	\end{equation*}
	In particular, $\DLambda_C^- \not \geq \LambdaR(A_D) - \Lambda(\AR_0)$.

	\section{Proof of Theorem \ref{thm:mainextreme}} \label{sec:proofmainextreme}
	
	Assume that $D\subseteq C$. For $\epsilon>0$ consider the following truncated coefficient:
	\begin{equation*}
		A_\epsilon = \begin{dcases}
			A_0 & \text{in } \Omega\setminus C \\
			\epsilon^{-1}\AR_D & \text{in } C.
		\end{dcases}
	\end{equation*}
	Let $u_\epsilon = u_f^{A_\epsilon}$ and recall that $A_0\in\Hsa(\Omega)$ is assumed in this theorem, meaning that $A_\epsilon\in\Hsa(\Omega)$. Let $\alpha>0$ and $\eta>0$ satisfy that $\AR_D\geq \alpha I$ in $\Omega$ and $\norm{\AI_D}_{\mathcal{H}(\Omega)}\leq \eta$. Then from Theorem~\ref{thm:generalmono} and $\epsilon\in(0,1]$ we have the estimate
	\begin{align*}
		\inner{f,\bigl[\LambdaR(A_D)-\Lambda(A_\epsilon)\bigr]f} &\geq \int_C \bigl[(\epsilon^{-1}-1)\AR_D - \AI_D(\AR_D)^{-1}\AI_D \bigr]\nabla u_\epsilon\cdot\overline{\nabla u_\epsilon}\,\di x \\
		&\geq \bigl[ (\epsilon^{-1}-1)\alpha - \tfrac{\eta^2}{\alpha}\bigr]\int_C \abs{\nabla u_\epsilon}^2\,\di x.
	\end{align*}
	Hence for $\epsilon \leq 1/(1 + (\eta/\alpha)^2)$ we have
	\begin{equation*}
		\inner{f,\bigl[\LambdaR(A_D)-\Lambda(A_\epsilon)\bigr]f} \geq 0.
	\end{equation*}
	Since $\Lambda(A_\epsilon) \to \Lambda_\emptyset^C$ in the operator norm as $\epsilon\to 0$ by \cite[theorem~13.3 and remark~13.4]{GJZ2025}, we conclude that 
	\begin{equation*}
		\LambdaR(A_D) \geq \Lambda_\emptyset^C.
	\end{equation*}
	From Theorem~\ref{thm:generalmono} we have $\Lambda(\AR_D) \geq \LambdaR(A_D)$ and from \cite[theorem~12.2]{GJZ2025} we have $\Lambda_C^\emptyset \geq \Lambda(\AR_D)$. So we conclude that
	\begin{equation*}
		\Lambda_C^\emptyset \geq \LambdaR(A_D).
	\end{equation*}
	
	Now assume that $D \not\subseteq C$, i.e.~that $D^\bullet\not\subseteq C$. Since both sets are closures of open sets and have connected complements, there is a relatively open set that connects $D^\bullet\setminus C$ with $\Gamma$ that we may use for localization. Specifically, there exist a relatively open connected set $U\subset \Omega\setminus C$ that intersects $\Gamma$ and an open ball $B\subset U \cap D$. Moreover, Assumption~\ref{assump:reconextreme}(iii) ensures that we can pick $B$ and $U$ such that we arrive at one of two cases:
	\begin{itemize}
		\item Case 1: $\AR_D - A_0$ is positive semidefinite in $U$ and uniformly positive definite in $B$.
		\item Case 2: $\AR_D - A_0$ is negative semidefinite in $U$ and uniformly negative definite in $B$.
	\end{itemize}
	In the first case we will prove that $\LambdaR(A_D) \not\geq \Lambda_\emptyset^C$, and in the second case we will prove that $\Lambda_C^\emptyset \not\geq \LambdaR(A_D)$. Together the two cases prove the contrapositive formulation of the assertion 
	\begin{equation*}
		\Lambda_C^{\emptyset} \geq \LambdaR(A_D) \geq \Lambda_{\emptyset}^C \quad \text{\emph{implies}} \quad D \subseteq C.
	\end{equation*}
	
	\subsection*{Case 1}
	
	Note that we have
	\begin{equation*}
		\LambdaR(A_D) - \Lambda_\emptyset^C = \bigl[\LambdaR(A_D) - \Lambda(\AR_D)\bigr] + \bigl[\Lambda(\AR_D) - \Lambda_\emptyset^C\bigr].
	\end{equation*}
	From Theorem~\ref{thm:generalmono} we have $\LambdaR(A_D) \leq \Lambda(\AR_D)$. From \cite[case 2 in proof of theorem~12.2]{GJZ2025} there exists $f_0\in L_\diamond^2(\Gamma)$ such that $\inner{f_0,\bigl[\Lambda(\AR_D) - \Lambda_\emptyset^C\bigr]f_0} < 0$, which implies
	\begin{equation*}
		\inner{f_0,\bigl[\LambdaR(A_D) - \Lambda_\emptyset^C\bigr]f_0} < 0,
	\end{equation*}
	i.e.~$\LambdaR(A_D) \not\geq \Lambda_\emptyset^C$.
	
	\subsection*{Case 2}
	
	Note that we have
	\begin{equation*}
		\Lambda_C^\emptyset - \LambdaR(A_D) = \bigl[\Lambda_C^\emptyset - \Lambda(A_0)\bigr] + \bigl[\Lambda(A_0) - \LambdaR(A_D)\bigr].
	\end{equation*}
	From Theorem~\ref{thm:locpot} there exists a sequence $(f_j)$ in $L^2_\diamond(\Gamma)$ such that $u_j = u_{f_j}^{A_0}$ satisfy
	\begin{equation*}
		\lim_{j\to\infty} \int_B \abs{\nabla u_j}^2\,\di x = \infty \qquad \text{and} \qquad \lim_{j\to\infty}\int_{\Omega\setminus U} \abs{\nabla u_j}^2\,\di x = 0.
	\end{equation*}
	Let
	\begin{align*}
		\textup{I}_j^1 &= \inner{f_j,\bigl[\Lambda_C^\emptyset - \Lambda(A_0)\bigr]f_j}, \\
		\textup{I}_j^2 &= \int_{\Omega\setminus U} \bigl[\AR_D - A_0 + \AI_D(\AR_D)^{-1}\AI_D\bigr]\nabla u_j\cdot\overline{\nabla u_j}\,\di x.
	\end{align*}
	Since $C\subset \Omega\setminus U$, and due to \cite[propositions~14.3 and~15.1]{GJZ2025}, then $\lim_{j\to\infty} \textup{I}_j^k = 0$ for $k\in\{1,2\}$.	
	
	Due to Assumption~\ref{assump:reconextreme}(ii), we may pick $U$ such that $\AI_D = 0$ in $U$. With that in mind, Theorem~\ref{thm:generalmono} gives
	\begin{align*}
		\inner{f_j,\bigl[\Lambda(A_0) - \LambdaR(A_D)\bigr]f_j} &\leq \int_U (\AR_D - A_0)\nabla u_j\cdot\overline{\nabla u_j}\,\di x + \textup{I}_j^2 \\
		&\leq -c\int_B\abs{\nabla u_j}^2\,\di x + \textup{I}_j^2,
	\end{align*}
	for some scalar $c>0$. In total we have
	\begin{equation*}
		\inner{f_j,\bigl[\Lambda_C^\emptyset - \LambdaR(A_D)\bigr]f_j} \leq -c\int_B\abs{\nabla u_j}^2\,\di x + \textup{I}_j^1 + \textup{I}_j^2 \to -\infty \text{ for } j\to\infty.
	\end{equation*}
	In particular, $\Lambda_C^\emptyset \not\geq \LambdaR(A_D)$.
	
	\subsection*{Acknowledgements}
	
	The authors are supported by grant 10.46540/3120-00003B from Independent Research Fund Denmark.
	
	\appendix
	\section{Towards improved monotonicity inequalities} \label{appA}
	
	In this section we prove inequalities in the style of Theorem~\ref{thm:generalmono}, but with optimal bounds in the sense that the lower and upper bounds are zero if the coefficients coincide (unlike in Theorem~\ref{thm:generalmono} for non-self-adjoint coefficients). We also prove that the additional terms that show up due to the non-self-adjointness can be controlled by the norm of $u_2$. Although, as will be apparent at the end of this appendix, an even more precise control in terms of $u_2$ will be needed in order to improve the monotonicity method in the non-self-adjoint setting, related to the use of localized potentials from Theorem~\ref{thm:locpot}.
	
	\begin{theorem}\label{thm:improvedmono}
		Let $A_1,A_2\in\mathcal{H}(\Omega)$ and $f\in L^2_\diamond(\Gamma)$. For $j\in\{1,2\}$ we denote $\Lambda_j = \Lambda(A_j)$ and $u_j = u_f^{A_j}$. Then
		\begin{align*}
			\inner{f,(\LambdaR_1-\LambdaR_2) f} &\geq \int_\Omega (\AR_2-\AR_1)\nabla u_2\cdot \overline{\nabla u_2}\,\di x + 2\imdel\int_\Omega \AI_1\nabla u_1\cdot\overline{\nabla u_2}\,\di x,  \\
			\inner{f,(\LambdaR_1-\LambdaR_2) f} &\leq \int_\Omega \AR_2(\AR_1)^{-1}(\AR_2-\AR_1)\nabla u_2\cdot\overline{\nabla u_2}\,\di x + 2\imdel \int_\Omega \AI_2\nabla u_1\cdot\overline{\nabla u_2}\,\di x.
		\end{align*}
	\end{theorem}
	\begin{proof}
		We start by noting that 
		\begin{align} \label{eq:mixed1}
			\redel \int_\Omega \AR_1\nabla u_1\cdot\overline{\nabla u_2}\,\di x &= \redel \int_\Omega (A_1 - \I \AI_1)\nabla u_1\cdot\overline{\nabla u_2}\,\di x \notag \\
			&= \redel \int_\Omega A_2\nabla u_2\cdot\overline{\nabla u_2}\,\di x - \redel \I\int_\Omega \AI_1\nabla u_1\cdot\overline{\nabla u_2}\,\di x \notag \\
			&= \int_\Omega \AR_2 \nabla u_2\cdot\overline{\nabla u_2}\,\di x + \imdel \int_\Omega \AI_1 \nabla u_1\cdot\overline{\nabla u_2}\,\di x.
		\end{align}
		By swapping the indices in \eqref{eq:mixed1}, and using that $\AR_2$ and $\AI_2$ are self-adjoint, we also have
		\begin{equation} \label{eq:mixed2}
			\redel \int_\Omega \AR_2\nabla u_1\cdot\overline{\nabla u_2}\,\di x = \int_\Omega \AR_1 \nabla u_1\cdot\overline{\nabla u_1}\,\di x - \imdel \int_\Omega \AI_2 \nabla u_1\cdot\overline{\nabla u_2}\,\di x.
		\end{equation}
		
		Using \eqref{eq:mixed1} and that $\AR_1$ is positive definite, we have
		\begin{align*}
			\int_\Omega (\AR_2 - \AR_1)\nabla u_2\cdot\overline{\nabla u_2}\,\di x + 2\imdel\int_\Omega \AI_1\nabla u_1\cdot\overline{\nabla u_2}\,\di x \hspace{-7cm}& \\
			&= \int_\Omega (\AR_2 - \AR_1)\nabla u_2\cdot\overline{\nabla u_2}\,\di x + 2\redel \int_\Omega \AR_1\nabla u_1\cdot\overline{\nabla u_2}\,\di x - 2\int_\Omega \AR_2 \nabla u_2\cdot\overline{\nabla u_2}\,\di x \\
			&\leq \int_\Omega (\AR_2 - \AR_1)\nabla u_2\cdot\overline{\nabla u_2}\,\di x + 2\redel \int_\Omega \AR_1\nabla u_1\cdot\overline{\nabla u_2}\,\di x - 2\int_\Omega \AR_2 \nabla u_2\cdot\overline{\nabla u_2}\,\di x\\
			&\phantom{={}} + \int_\Omega \AR_1\nabla(u_1-u_2)\cdot\overline{\nabla(u_1-u_2)}\,\di x\\
			&= \int_\Omega \AR_1\nabla u_1\cdot\overline{\nabla u_1}\,\di x - \int_\Omega \AR_2 \nabla u_2\cdot\overline{\nabla u_2}\,\di x.
		\end{align*}
		That is, 
		\begin{equation*}
			\inner{f,(\LambdaR_1-\LambdaR_2)f} \geq \int_\Omega (\AR_2 - \AR_1)\nabla u_2\cdot\overline{\nabla u_2}\,\di x + 2\imdel\int_\Omega \AI_1\nabla u_1\cdot\overline{\nabla u_2}\,\di x.
		\end{equation*}
		
		Using \eqref{eq:mixed2}, and that $\AR_2(\AR_1)^{-1}(\AR_2 - \AR_1)$ is self-adjoint, we obtain
		\begin{align*}
			\int_\Omega \AR_2(\AR_1)^{-1}(\AR_2 - \AR_1) \nabla u_2\cdot\overline{\nabla u_2}\,\di x + 2\imdel \int_\Omega \AI_2\nabla u_1\cdot\overline{\nabla u_2}\,\di x \hspace{-9.5cm}& \\
			&= \int_\Omega \AR_2(\AR_1)^{-1}(\AR_2 - \AR_1) \nabla u_2\cdot\overline{\nabla u_2}\,\di x + 2\int_\Omega \AR_1 \nabla u_1\cdot\overline{\nabla u_1}\,\di x - 2\redel \int_\Omega \AR_2\nabla u_1\cdot\overline{\nabla u_2}\,\di x \\
			&\geq \int_\Omega \AR_2(\AR_1)^{-1}(\AR_2 - \AR_1) \nabla u_2\cdot\overline{\nabla u_2}\,\di x + 2\int_\Omega \AR_1 \nabla u_1\cdot\overline{\nabla u_1}\,\di x - 2\redel \int_\Omega \AR_2\nabla u_1\cdot\overline{\nabla u_2}\,\di x \\
			&\phantom{={}} - \int_\Omega \bigl\lvert (\AR_1)^{1/2}\nabla u_1 - (\AR_1)^{-1/2}\AR_2\nabla u_2 \bigr\rvert^2\,\di x \\
			&= \int_\Omega \AR_1\nabla u_1\cdot\overline{\nabla u_1}\,\di x - \int_\Omega \AR_2 \nabla u_2\cdot\overline{\nabla u_2}\,\di x.
		\end{align*}
		That is, 
		\begin{equation*}
			\inner{f,(\LambdaR_1-\LambdaR_2)f} \leq \int_\Omega \AR_2(\AR_1)^{-1}(\AR_2 - \AR_1) \nabla u_2\cdot\overline{\nabla u_2}\,\di x + 2\imdel \int_\Omega \AI_2\nabla u_1\cdot\overline{\nabla u_2}\,\di x. \qedhere
		\end{equation*}
	\end{proof}
	
	What remains is to establish estimates of
	\begin{equation} \label{eq:finalterm}
		\imdel \int_\Omega \AI_j\nabla u_1\cdot\overline{\nabla u_2}\,\di x
	\end{equation}
	with respect to $u_2$, in order to be able to combine such inequalities with localization as in Theorem~\ref{thm:locpot}.
	
	Let $N(A) \colon f\mapsto u_f^A$ be the solution operator for the PDE problem with coefficient $A$, then using Taylor's theorem (in Banach spaces) with integral remainder term, we have
	\begin{equation*}
		u_1 = u_2 + \int_0^1 \textup{D}\!N\bigl(A_2 + t(A_1-A_2);A_1-A_2\bigr)f\,\di t.
	\end{equation*}
	As $\AI_j$ is self-adjoint and due to the imaginary part in \eqref{eq:finalterm}, we have that 
	\begin{equation*} 
		\imdel \int_\Omega \AI_j\nabla u_2\cdot\overline{\nabla u_2}\,\di x = 0,
	\end{equation*}
	so we only have to consider the remainder term from Taylor's theorem. In the following we denote $A_t = A_2 + t(A_1-A_2)$, $u_t = u_f^{A_t}$, and let $w_t\in H^1_\diamond(\Omega)$ be the unique solution to
	\begin{equation} \label{eq:varderiv}
		\int_\Omega A_t\nabla w_t\cdot\overline{\nabla v}\,\di x = \int_\Omega (A_2-A_1)\nabla u_t\cdot\overline{\nabla v}\,\di x,
	\end{equation}
	for all $v\in H^1_\diamond(\Omega)$. From \cite[theorem 3.3]{Garde2022c}, we have
	\begin{equation*}
		\textup{D}\!N\bigl(A_2 + t(A_1-A_2);A_1-A_2\bigr)f = w_t.
	\end{equation*}
	Let $\bigcup_{t\in[0,1]}\suppm(\nabla w_t)\subseteq W$ for some measurable $W\subseteq \Omega$, and let $M = \suppm(\AI_j)$. Then \eqref{eq:finalterm} becomes
	\begin{align} \label{eq:finaltemest1}
		\abs{\imdel \int_\Omega \AI_j\nabla u_1\cdot\overline{\nabla u_2}\,\di x} &= \abs{\imdel \int_0^1\int_{W\cap M} \AI_j\nabla w_t\cdot\overline{\nabla u_2}\,\di x\,\di t} \notag \\
		&\leq \norm{\AI_j}_{\mathcal{H}(W\cap M)}\Bigl(\int_0^1\norm{\nabla w_t}_{L^2(W)^d}\,\di t \Bigr)\norm{\nabla u_2}_{L^2(W\cap M)^d}.
	\end{align}
	If $\AR_1 \geq \alpha I$ and $\AR_2 \geq \alpha I$ for some $\alpha>0$, we have $\AR_t \geq \alpha I$ as $A_t$ is a convex combination of $A_1$ and $A_2$. Hence, from \eqref{eq:varderiv} we have
	\begin{align*}
		\alpha \norm{\nabla w_t}_{L^2(W)^d}^2 &\leq \redel \int_\Omega A_t\nabla w_t\cdot\overline{\nabla w_t}\,\di x\\
		&= \redel \int_W (A_2-A_1)\nabla u_t\cdot\overline{\nabla w_t}\,\di x \\
		&\leq \norm{A_1-A_2}_{\mathcal{H}(W)}\norm{\nabla w_t}_{L^2(W)^d}\norm{\nabla u_t}_{L^2(W)^d},
	\end{align*}
	which implies
	\begin{equation*}
		\norm{\nabla w_t}_{L^2(W)^d} \leq \alpha^{-1}\norm{A_1-A_2}_{\mathcal{H}(W)}\norm{\nabla u_t}_{L^2(W)^d}.
	\end{equation*}
	So \eqref{eq:finaltemest1} becomes
	\begin{equation*}
		\abs{\imdel \int_\Omega \AI_j\nabla u_1\cdot\overline{\nabla u_2}\,\di x} \leq \alpha^{-1}\norm{\AI_j}_{\mathcal{H}(W\cap M)}\norm{A_1-A_2}_{\mathcal{H}(W)}\Bigl(\int_0^1\norm{\nabla u_t}_{L^2(W)^d}\,\di t\Bigr)\norm{\nabla u_2}_{L^2(W\cap M)^d}.
	\end{equation*}
	From the Lax--Milgram lemma, we have
	\begin{equation*}
		\norm{\nabla u_t}_{L^2(W)^d} \leq \norm{\nabla u_t}_{L^2(\Omega)^d} \leq \alpha^{-1}C_1\norm{\nu\cdot (A_2\nabla u_2)}_{H^{-1/2}(\partial\Omega)},
	\end{equation*}
	since $\AR_t \geq \alpha I$ for all $t\in[0,1]$ and the Neumann condition coincides for all $u_t$ on $\partial\Omega$ (and with $t=0$ corresponding to $u_2$). Here $C_1$ is the operator norm of the Dirichlet trace operator in $\mathscr{L}(H^1_\diamond(\Omega),H^{1/2}(\partial\Omega))$, where $H^1_\diamond(\Omega)$ is equipped with the norm $v\mapsto\norm{\nabla v}_{L^2(\Omega)^d}$. We moreover have that
	\begin{equation*}
		\norm{\nu\cdot (A_2\nabla u_2)}_{H^{-1/2}(\partial\Omega)} \leq C_2\norm{A_2}_{\mathcal{H}(\Omega)}\norm{\nabla u_2}_{L^2(\Omega)^d},
	\end{equation*}
	where $C_2$ is the operator norm of the mapping $F\mapsto \nu\cdot F|_{\partial\Omega}$ (trace of a normal component of a vector function) in $\mathscr{L}(H_{\textup{div}}(\Omega),H^{-1/2}(\partial\Omega))$. Hence, $C = C_1C_2$ only depends on $\Omega$, and in total we have
	\begin{equation*}
		\abs{\imdel \int_\Omega \AI_j\nabla u_1\cdot\overline{\nabla u_2}\,\di x} \leq \tfrac{C}{\alpha^2}\norm{\AI_j}_{\mathcal{H}(W\cap M)}\norm{A_1-A_2}_{\mathcal{H}(W)}\norm{A_2}_{\mathcal{H}(\Omega)}\norm{\nabla u_2}_{L^2(\Omega)^d}\norm{\nabla u_2}_{L^2(W\cap M)^d}.
	\end{equation*}
	So it is indeed possible to estimate \eqref{eq:finalterm} solely in terms of $u_2$, and such that the bound vanishes in case $A_1 = A_2$ or if $\AI_j = 0$. However, we still have the norm of $u_2$ on all of $\Omega$, which remains insufficient to avoid additional conditions on a background conductivity in inclusion detection. 
	
	\bibliographystyle{plain}

\end{document}